\documentclass{amsart}

\usepackage{amsmath,amsfonts,amssymb,amsthm}
\usepackage{graphicx}
\usepackage{caption}
\usepackage{subcaption}
\usepackage[colorlinks=true, allcolors=blue]{hyperref}
\usepackage{tikz}
\usetikzlibrary{calc}

\pgfdeclarelayer{background}
\pgfdeclarelayer{foreground}
\pgfsetlayers{background,foreground}
\newtheorem{thm}{Theorem}

\newtheorem{prop}[thm]{Proposition}

\newtheorem{claim}[thm]{Claim}

\newcommand{\sat}{{\rm sat}}
\newcommand{\F}{{\mathcal F}}
\newcommand{\C}{{\mathcal C}}
\newcommand{\D}{{\mathcal D}}
\newcommand{\power}{{\mathcal P}}
\newcommand{\Q}{{\mathcal Q}}
\newcommand{\B}{{\mathcal B}}
\newcommand{\chain}{\C^{\flat}}

\title{Induced saturation of the poset $2C_2$}
\author{Ryan R. Martin}
\address{Department of Mathematics, Iowa State University, Ames, Iowa, USA}
\author{Nick Veldt}
\email{rymartin@iastate.edu, mnveldt@iastate.edu}
\begin{document}
\begin{abstract}
    Given a set $X$, the power set $\power(X)$, and a finite poset $(P,\leq_P)$, a family $\F\subseteq\power(X)$ is said to be induced-$P$-free if there is no injection $\varphi: P\rightarrow\F$ such that $\varphi(p)\subseteq\varphi(q)$ if and only if $p\leq_{P} q$ for every $p,q \in P$. The family $\F$ is \emph{induced-$P$-saturated} if it is maximal with respect to being induced-$P$-free. If $n=|X|$, then the size of the smallest induced-$P$-saturated family in $\power(X)$ is denoted $\sat^*(n,P)$. 

    The poset $2C_2$ is two incomparable 2-chains (the Hasse diagram is two vertex-disjoint edges) and Keszegh, Lemons, Martin, P\'alv\"olgyi, and Patk\'os proved that $n+2\leq\sat^*(n,2C_2)\leq 2n$ and gave one isomorphism class of an induced-$2C_2$-saturated family that achieves the upper bound. 

    We show that the lower bound can be improved to $3n/2 + 1/2$ by examining the necessary structure of a saturated family. In addition, we provide many examples of induced-$2C_2$-saturated families of size $2n$ in $\power(X)$ where $|X|=n$. 
    
\end{abstract}
\maketitle
\section{Introduction}
Given a set $Y$, we say that a collection $\F \subseteq Y$ is \emph{saturated} with respect to a property $\Q$ if $\F$ does not exhibit the property $\Q$, but $\F\cup\{y\}$ has $\Q$ for all $y \in Y \setminus \F$. The \emph{saturation number} is the number of elements in the smallest saturated subcollection of $Y$. 

In extremal poset theory we mostly consider the environment $Y$ to be the $n$-dimensional Boolean lattice $\B_n$, and the property $\Q$ to be the existence of a specific poset as a (weak) subposet or as an induced subposet. The poset $(P, \leq_P)$ is a \emph{weak subposet} of poset $(P',\leq_{P'})$ if there is an injection $\varphi: P\rightarrow P'$ such that $\phi(p_1)\leq_{P'}\phi(p_2)$ whenever $p_1\leq_{P}p_2$. The poset $(P, \leq_P)$ is an \emph{induced subposet} of $(P',\leq_{P'})$ if there is an injection $\varphi: P\rightarrow P'$ such that $\phi(p_1)\leq_{P'}\phi(p_2)$ if and only if $p_1\leq_{P}p_2$. 

Saturation has been studied in graph theory beginning with Erd\H{o}s, Hajnal, and Moon~\cite{EHM}, see also a dynamic survey by Currie, Faudree, Faudree and Schmitt~\cite{FFS}. The notion of induced saturation in graphs was introduced by the first author and Smith~\cite{MS} and subsequently studied~\cite{BESYY,ACs} but the notion of induced saturation is not very natural in graphs and, in this context, requires the introduction of so-called trigraphs, defined by Chudnovsky~\cite{Chudnovsky}.

In posets, induced saturation is natural when the underlying poset is a Boolean lattice. A family $\F \subseteq \B_n$ is said to be \emph{induced-$P$-saturated} if $\F$ does not contain $P$ as an induced subposet, but for any element $S \in \B_n$, $S \notin F$, $\F \cup \{S\}$ contains $P$ as an induced subposet. The \emph{induced saturation number} of $P$, denoted $\sat^*(n,P)$, is the number of elements in the smallest induced-$P$-saturated family $\F \subseteq \B_n$. In this context, the idea was introduced by Ferrara, Kay, Kramer, Martin, Reiniger, Smith, and Sullivan~\cite{FKKMRSS}, and has been studied for a wide variety of different subposets in that setting~\cite{GERB, MNS,KLMPP,MSW,FPSST,BGJJ,BGIJ,IvanButterfly,IvanDiamond,Liu}.

In this paper, we examine the induced saturation number of the poset $2C_2$, which is the poset that consists of four elements $\{A, A', B, B'\}$ among which $A \subsetneq A'$, $B \subsetneq B'$ are the only relations. See Figure~\ref{fig:2C2}.

\begin{figure}[h]
\begin{tikzpicture}[scale=1,vertex/.style={draw=black, very thick, fill=white, circle, minimum width=2pt, inner sep=2pt, outer sep=1pt},edge/.style={very thick}]
    \coordinate (zero) at (0,0);
    \begin{pgfonlayer}{foreground}
    \node[vertex] (A) at ($(zero)+(-0.5,-0.5)$) {};
    \node[left,xshift=-1pt] at (A) {$A$};
    \node[vertex] (Ap) at ($(zero)+(-0.5,0.5)$) {};
    \node[left,xshift=-1pt] at (Ap) {$A'$};
    \node[vertex] (B) at ($(zero)+(0.5,-0.5)$) {};
    \node[right,xshift=1pt] at (B) {$B$};
    \node[vertex] (Bp) at ($(zero)+(0.5,0.5)$) {};
    \node[right,xshift=1pt] at (Bp) {$B'$};
    \end{pgfonlayer}
    \begin{pgfonlayer}{background}
    \draw[edge] (A) -- (Ap);
    \draw[edge] (B) -- (Bp);
    \end{pgfonlayer}
\end{tikzpicture}
\caption{Hasse diagram of $2C_2$.}
\label{fig:2C2}
\end{figure}
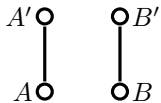

The saturation number of $2C_2$ on the $n$-dimensional Boolean lattice is denoted $\sat^*(n,2C_2)$.

\begin{thm}[Keszegh, Lemons, Martin, P\'alv\"olgyi, Patk\'os~\cite{KLMPP}]
    For all $n\geq 2$, $n+2 \leq \sat^*(n,2C_2)\leq 2n$.
\end{thm}
The upper bound comes from the following construction. Consider a family $\F$ consisting of a full chain and the singletons, for example, $\F = \{\emptyset,\{1\},\{1,2\},\ldots,[n],\linebreak\{2\},\{3\},\ldots,\{n\}\}$. It is easy to see that there is no induced copy of $2C_2$ and this is the entire Boolean lattice if $n=2$. Now let $n\geq 3$ consider any set $S \notin \F$. Let $l$ be the least element not in $S$, and let $m$ be the greatest element in $S$. Note that $l < m$, and that $3 \leq m$.

If $l \in \{1,2\}$, then $\{\{l\},\{1,2\},\{m\},S\}$ forms an induced $2C_2$ in $\F \cup \{S\}$. If $l \geq 3$, then $\{\{l\},[l],\{m\},S\}$ forms an induced $2C_2$.

The lower bound in Theorem 1 follows from the fact
\cite[Theorem 3.11]{KLMPP} that an induced-$2C_2$-saturated family
must contain a maximal chain. In Theorem ~\ref{thm:KLMPP}, we will go through this proof, as it is crucial to further tightening the lower bound.

We present a new result on the lower bound of $\sat^*(n,2C_2)$:
\begin{thm}
    $\frac{3}{2}n+\frac{1}{2} \leq \sat^*(n,2C_2)\leq 2n$. \label{thm:main}
\end{thm}

The improvement of the lower bound comes from a consideration of elements that are near to the required maximal chain.

We also employed a brute force algorithm to find saturated families and through computer search, verified that $\sat^*(n,2C_2) = 2n$ for all $n\leq 6$. In Appendix~\ref{sec:constructions}, we list all of the extremal constructions for $n\leq 5$.
\subsection{Terms and notation}

For a positive integer $n$, $[n]$ denotes the set $\{1,\ldots,n\}$. For consistency, $[0]=\emptyset$.  
The members of a poset are called \emph{vertices}. Let $A$ and $B$ be vertices of a poset $(P,\leq_P)$. If either $A\leq_P B$ or $A\geq_P B$, then we write $A\sim B$ and say that $A$ is \emph{related} to $B$. If $A\not\leq_P B$ and $A\not\geq_P B$, then we write $A\parallel B$ and say $A$ is \emph{unrelated} to $B$ or that $A$ and $B$ are \emph{incomparable}. To save time, we occasionally use $\{A,B\}\parallel\{C,D\}$ to denote that each of $A$ and $B$ is incomparable to each of $C$ and $D$. The notation $A\parallel\{C,D\}$, of course, means that $A$ is incomparable to each of $C$ and $D$. 

For any positive integer $n$, let $\B_n$ denote the \emph{Boolean lattice of dimension $n$}. Given $A\subseteq B$ in $\B_n$, the \emph{interval of $A$ and $B$} is defined to be $[A,B]=\{S : A\subseteq S\subseteq B\}$. The \emph{half-open intervals of $A$ and $B$} are $(A,B]=\{S : A\subsetneq S\subseteq B\}$ and $[A,B)=\{S: A\subseteq S\subsetneq B\}$. A \emph{chain} of length $m$ is a group of elements $A_1, \ldots, A_m$ such that $A_1 \subsetneq \ldots \subsetneq A_m$. A \emph{maximal chain} is a chain that is maximal with respect to length. In particular, any maximal chain in $\B_n$ has $n+1$ distinct vertices. 

For any set $S \in \B_n$, we use $\overline{S} := \{i \in [n] : i \notin S\}$ to represent the set complement. The vertices of $\B_n$ that contain only one element are called \emph{singletons}. The vertices of $\B_n$ that contain all but one element are called \emph{anti-singletons}.

\subsection{Organization} In Section~\ref{sec:maximal}, we state an important lemma from~\cite{KLMPP} that establishes the fact that an induced-$2C_2$-saturated family must contain a maximal chain and we provide a proof for completeness that uses some of the terminology we use in this paper. This is important for introducing some ideas that we use later. In Section~\ref{sec:mainproof}, we prove Theorem~\ref{thm:main}.  In Section~\ref{sec:examples}, we discuss the proliferation of induced-$2C_2$-saturated families in $\B_n$ of size $2n$, for small values of $n$, including a new family of valid constructions. In Section~\ref{sec:conc} we finish with some concluding remarks on our observations. In Appendix~\ref{sec:constructions}, we give constructions of such families for $n\leq 5$. 

\section{Maximal chains}
\label{sec:maximal}

The most important starting point for our result is a theorem by Keszegh et al., demonstrating that we may assume that the families we consider must have a maximal chain.
\begin{thm}[Keszegh et al.~\cite{KLMPP}]\label{thm:KLMPP}
    If $\F \subseteq 2^{[n]}$ is induced-$2C_2$-saturated, then $\F$ contains a maximal chain in $2^{[n]}$. \label{thm:KLMPP}
\end{thm}

For a family $\F$ and a set $S$ we define the \emph{open downset} $\D_\F(S) := \{F \in \F : F \subsetneq S\}$. That is, $\D_\F(S)$ is the family of all subsets of $S$, not including $S$ itself. 

\begin{prop}[Keszegh et al.~\cite{KLMPP}]
    If $\F$ is induced-$2C_2$-free, then for any $F,F' \in \F$, one of three possibilities hold:
    \begin{itemize}
        \item $\D_\F(F) \subsetneq \D_\F(F')$
        \item $\D_\F(F') \subsetneq \D_\F(F)$
        \item $\D_\F(F) = \D_\F(F')$
    \end{itemize}
    \label{prop:downsets}
\end{prop}
\begin{proof}
    If there exists $G \in \D_\F(F) \setminus \D_\F(F')$ and $G' \in \D_\F(F') \setminus \D_\F(F)$, then $G\subsetneq F,G'\subsetneq F'$ is an induced $2C_2$.
\end{proof}

\begin{proof}[Proof of Theorem~\ref{thm:KLMPP}]
    Let $\F$ be a induced-$2C_2$-saturated family. Each of the sets $\emptyset$ and $[n]$ must belong to $\F$, as they are comparable to every other set, and so cannot be part of a $2C_2$.

    Now, we order the sets in $\F$ by their downsets. We define the preorder $F \leq F'$ if $\D_\F(F) \subseteq \D_\F(F')$. Note that this is not necessarily a partial order because if $\{1\},\{2\}\in\F$, then $\D_\F\bigl(\{1\}\bigr)=\D_\F\bigl(\{2\}\bigr)=\emptyset$. Now, choose an arbitrary linear order $\emptyset = F_m \leq \cdots \leq F_2 \leq F_1 = [n]$ that respects the preorder. 

    Now, we introduce a second family of sets, defined as $G_j = \bigcap_{i=1}^j F_i$ for any $j = 1,2,\ldots,m$. Note that $G_j \supseteq G_{j+1}$ for $j = 1,2,\ldots,m-1$.
    
\begin{claim}
    For any $h = 1,2,\ldots, m$, we have $\D_\F(G_h) \subseteq \D_\F(F_h) \subseteq \D_\F(G_h) \cup \{G_h\}$.
    \label{claim:order}
\end{claim}

\begin{proof}
    By definition, $G_h = \bigcap_{i=1}^h F_i \subseteq F_h$, and so any subset of $G_h$ will be a subset of $F_h$, yielding the first relation.

    Then observe that any element $F$ in $\D_\F(F_h)$ will also be a subset of $F_i$ for $i \leq h$, due to the defined ordering, and so will be a subset of $G_h$ as well. The only element of $\D_\F(F_h)$ that is not contained in $\D_\F(G_h)$, then, would have to be $G_h$ itself. This concludes the proof of Claim~\ref{claim:order}.
\end{proof}

\begin{claim}
    For any $j = 1,2,\ldots, m-1$, if $G_{j+1} \subseteq S \subsetneq G_j$, then $S \in \F$.
    \label{claim:gap}
\end{claim}

    \begin{proof}
    Suppose there is an $S\not\in\F$ and $j\in\{1,2,\ldots,m-1\}$ with $G_{j+1} \subseteq S \subsetneq G_j$. Then, by the definition of saturation, adding $S$ to $\F$ creates an induced copy of $2C_2$. Thus, there must exist a pair $A \subsetneq B$ in $\F$ incomparable to $S$ such that $A = F_i$ and $B = F_k$ for some $k < i$. If $k \leq j$, then $S \subsetneq G_j \subseteq F_k = B$ gives a contradiction. Thus $k \geq j+1$. Applying Claim~\ref{claim:order} to $h = j+1$ gives
    $$ A \in \D_\F(F_k) \subseteq \D_\F(F_{j+1}) \subseteq \D_\F(G_{j+1}) \cup \{G_{j+1}\}\subseteq \D_{\F}(S) \cup \{S\}, $$
    contradicting the assumption that $A$ and $S$ are incomparable. This concludes the proof of Claim~\ref{claim:gap}. 
    \end{proof}

 Since $\emptyset=G_m\subseteq \cdots \subseteq G_2\subseteq G_1=[n]$ and $G_1\in\F$, we can use induction on $j \in \{1,2,\ldots,m-1\}$ to prove the existence of a maximal chain: Assume that the interval $[G_{j},G_1]$ contains a maximal chain of members of $\F$. If $G_{j+1} = G_j$, then $[G_{j+1},G_1]$ will trivially also contain such a maximal chain. If $G_{j+1} \neq G_j$, then by Claim~\ref{claim:gap}, every member of the half-open interval $[G_{j+1},G_j)$ is in $\F$, and therefore $[G_{j+1},G_1]$ will contain a maximal chain of members of $\F$. This concludes the proof of Theorem~\ref{thm:KLMPP}.
\end{proof}

The existence of a maximal chain already gives a lower bound of $n+1$, and since all those elements are related, at minimum we need one more element to be capable of forcing a $2C_2$, giving the bound of $n+2$ from~\cite{KLMPP}. To improve the lower bound further, we use a counting argument on elements that are more difficult to account for than others.

\section{Proof of Theorem~\ref{thm:main}}
\label{sec:mainproof}

Let $\F \subseteq \B_n$ be an induced-$2C_2$-saturated family. By Theorem~\ref{thm:KLMPP}, without loss of generality, $\F$ contains the sets $C_0, C_1, C_2, \ldots, C_n$, where $C_i=[i]$ for $i=0,1,\ldots,n$. Let $\chain=\{C_0,C_1,C_2,\ldots,C_n\}$. The $n-1$ elements of the form $S_i = [i-1] \cup \{i+1\}$ for $i = 1,2,\ldots, n-1$ are called \emph{shackles}. We think of these as the elements closest to the chain. Note that $S_i$ is related to every element of $\chain$ other than $C_i$. 

Shackles are distinctive in that for any shackle $S_i$, any $T\not\in\chain$ and any $C_j,C_k\in\chain$, they cannot create an induced copy of $2C_2$. This is because $C_j$ and $C_k$ are related to each other and at least one of them must be related to $S_i$ also. This characteristic will allow us to guarantee a certain amount of non-shackle elements in $\F$ to improve the lower bound.

First, however, we establish a useful claim about non-shackle elements of $\F$.

\begin{claim}
    Let $\F$ be induced-$2C_2$-saturated, with maximal chain $\chain$, and let there exist $A,B \in \F\setminus\chain$ such that $B \subsetneq A$. Then one of the following occurs:
    \begin{itemize}
        \item There is at least one $j\in\{1,\ldots,n-1\}$ such that $B\subsetneq C_j\subsetneq A$.
        \item There is a unique $j\in\{1,\ldots,n-1\}$ such that $B\subseteq S_j\subseteq A$ and $C_j\parallel\{A,B\}$.
    \end{itemize}
    \label{claim:two_elements}
\end{claim}\begin{figure}[h]
\hfill
\begin{subfigure}[t]{0.4\textwidth}
\begin{tikzpicture}[scale=1,vertex/.style={draw=black, very thick, fill=white, circle, minimum width=2pt, inner sep=2pt, outer sep=1pt},edge/.style={very thick}]
    \coordinate (zero) at (-3.0,0);
    \begin{pgfonlayer}{foreground}
    \node[vertex] (0) at ($(zero)+(0,-2)$) {};
    \node[left,xshift=-2pt] at (0) {$\emptyset$};
    \node[vertex] (1) at ($(zero)+(0,2)$) {};
    \node[left,xshift=-2pt] at (1) {$[n]$};
    \node[vertex] (B) at ($(zero) +(1.5,-1)$) {};
    \node[below,yshift=-2pt] at (B) {$B$};
    \node[vertex] (A) at ($(zero) +(1.5,1)$) {};
    \node[above,yshift=2pt] at (A) {$A$};
    \node[vertex] (C) at ($(zero) + (-0.6,0)$) {};
    \node[left,xshift=-3pt] at (C) {$C_{j}$};
    \end{pgfonlayer}
    
    \begin{pgfonlayer}{background}
    \draw[edge, line width = 2.8pt] (0) to [out = 120, in = 275] (C);
    \draw[edge, line width = 2.8pt] (C) to [out = 85, in = 240] (1);

    \draw[edge] (C) -- (A);
    \draw[edge] (B) -- (A);
    \draw[edge] (B) -- (C);
    
    \end{pgfonlayer}
    \end{tikzpicture}
    \caption{There exists at least one $C_j$ such that $B \subsetneq C_j \subsetneq A$.}
    \label{fig:shackle1}
\end{subfigure} \hfill
\begin{subfigure}[t]{0.4\textwidth}
\begin{tikzpicture}[scale=1,vertex/.style={draw=black, very thick, fill=white, circle, minimum width=2pt, inner sep=2pt, outer sep=1pt},edge/.style={very thick}]
    \begin{pgfonlayer}{foreground}
    \coordinate (one) at (2.5,0);
    \node[vertex] (02) at ($(one)+(0,-2)$) {};
    \node[left,xshift=-2pt] at (02) {$\emptyset$};
    \node[vertex] (12) at ($(one)+(0,2)$) {};
    \node[left,xshift=-2pt] at (12) {$[n]$};
    \node[vertex] (A2) at ($(one)+(1.5,1)$) {};
    \node[above,yshift=2pt] at (A2) {$A$};
    \node[vertex] (B2) at ($(one) +(1.5,-1)$) {};
    \node[below,yshift=-2pt] at (B2) {$B$};
    \node[vertex] (C2) at ($(one) + (-0.6,0)$) {};
    \node[left,xshift=-3pt] at (C2) {$C_{j}$};
    \node[vertex] (Cplus) at ($(one) + (-0.55,0.6)$) {};
    \node[left,xshift=-3pt] at (Cplus) {$C_{j+1}$};
    \node[vertex] (Cminus) at ($(one) + (-0.55,-0.6)$) {};
    \node[left,xshift=-3pt] at (Cminus) {$C_{j-1}$};
    \node[vertex] (S2) at ($(one) + (0.2,0)$) {};
    \node[right,xshift=3pt] at (S2) {$S_{j}$};
    \end{pgfonlayer}
    \begin{pgfonlayer}{background}

        \draw[edge, line width = 2.8pt] (02) to [out = 120, in = 280] (Cminus);
        \draw[edge, line width = 2.8pt] (Cminus) to [out = 98, in = 273] (C2);
        \draw[edge, line width = 2.8pt] (C2) to [out = 87, in = 263] (Cplus);
    \draw[edge, line width = 2.8pt] (Cplus) to [out = 80, in = 240] (12);

    \draw[edge] (S2) -- (B2);
    \draw[edge] (S2) -- (A2);
    \draw[edge] (B2) -- (A2);
    \draw[edge] (S2) -- (Cplus);
    \draw[edge] (S2) -- (Cminus);
    \end{pgfonlayer}
\end{tikzpicture}

\caption{There exists a unique $C_j$ such that $C_j \parallel \{A,B\}$. Note that $A = S_j$ or $B = S_j$ are possibilities here.}
\label{fig:shackle2}
\end{subfigure}
\caption{The two types of interaction with the maximal chain of a related pair in Claim~\ref{claim:two_elements}.}
\end{figure}
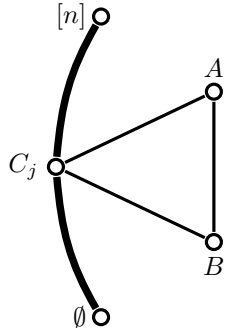
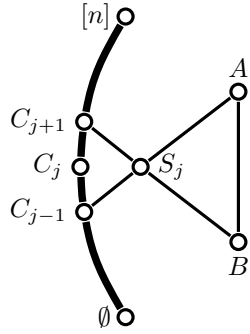

\begin{proof}
    Let $j_A$ be the largest integer such that $C_{j_A}\subsetneq A$ and let $j_B$ be the smallest integer such that $C_{j_B}\supsetneq B$.

    If $j_A\geq j_B$, then it is the case that $B\subsetneq C_{j_B}\subseteq C_{j_A}\subsetneq A$, which is the first condition, by setting $j$ to be, say, $j_B$. See Figure~\ref{fig:shackle1}.

    If $j_B-j_A=1$, then $B\subsetneq C_{j_B}=[j_B]$ but $B\not\subseteq C_{j_B-1}=[j_B-1]$. Thus $j_B\in B$. But then $j_B\in A$ also and since $[j_B-1]=[j_A]\subsetneq A$ it must be that $[j_B]\subsetneq A$, contradicting the maximality of $j_A$. 

    If $j_B-j_A=2$, let $j=j_A+1=j_B-1$. By the definition of $j_A$, $A\supsetneq C_{j-1}$ but $A\not\supseteq C_j$, thus $j\not\in A$ and hence $j\not\in B$. Further, by the definition of $j_B$, $B\subsetneq C_{j+1}$ but $B\not\subseteq C_j$, thus $j+1\in B$ and $j+1\in A$. Consequently, $B\subseteq S_j=[j-1]\cup\{j+1\}\subseteq A$, which is the second condition. See Figure~\ref{fig:shackle2}.

    If $j_B-j_A\geq 3$, then $\{C_{j_A+1},C_{j_B-1}\}\parallel\{B,A\}$, hence there is a $2C_2$, a contradiction.
\end{proof}

For a family $\F$ to be induced-$2C_2$-saturated, it must be true that every element not in $\F$ must form an induced $2C_2$ with some subset of $\F$. The shackles, requiring more than one element for this, are then a natural target for consideration.

For a shackle $S_i\not\in\F$, we say that an element $A\in\F$ \emph{witnesses} $S_i$ if $\F \cup \{S_i\}$ contains an induced $2C_2$ that includes $A$. So, we consider the methods in which $S_i$ can be witnessed.

\begin{claim}\label{cl:case breakdown}
A given shackle $S_i\not\in\F$, $i\in\{1,\ldots,n-1\}$ can be witnessed in one of 3 cases:
\begin{enumerate}
    \item $\exists\; A,B \in \F\setminus\chain: \hspace{13pt}A \,\parallel\, B,\; A\sim C_i,\; B\sim S_i,\; \{A,C_i\}\parallel \{B,S_i\}$. \label{it:1}
    \item $\exists\; A,B \in \F\setminus\chain: \hspace{13pt}A \supsetneq B,\; \exists \; j\neq i~\textrm{\rm such that}~\{A,B\} \parallel \{C_j,S_i\}$. \label{it:2}
    \item $\exists\; A,B,D \in \F\setminus\chain: A \supsetneq B,\; D \sim S_i,\; \{A,B\} \parallel \{D,S_i\}$. \label{it:3}
\end{enumerate}
Furthermore, in Case~(\ref{it:2}), the value of $j$ must be one of $\{i-1,i+1\}$ and $B\subseteq S_j\subseteq A$. 
\end{claim}

\begin{proof}
    Upon adding $S_i$ to $\F$, there must be an induced copy of $2C_2$ using $S_i$. No such $2C_2$ can contain two elements of $\chain$. Therefore, there must be (at least) two elements in $\F\setminus\chain$ that form this copy of $2C_2$. 

    If the witness uses an element of $\chain$ unrelated to $S_i$, then it must be $C_i$ and the other elements must be $A,B\in\F\setminus\chain$ such that $A\sim C_i$ and $B\sim S_i$ and $\{A,C_i\}\parallel \{B,S_i\}$. This is Case~(\ref{it:1}).

    If the witness uses an element of $\chain$ related to $S_i$, then call it $C_j$. The other edge of the $2C_2$ must contain related elements $B\subsetneq A$ in $\F\setminus \chain$. 
    
    As to the ``furthermore'' part of the proof, observe that since $A$ and $B$ must be unrelated to $C_j$, then by Claim~\ref{claim:two_elements}, we have $B \subseteq S_j \subseteq A$. As the shackle $S_j$ is related to every chain element except $C_j$, and thus related to all shackles except $S_{j-1}$ and $S_{j+1}$, we must have $i \in \{j-1,j+1\}$ or else one of $\{A,B\}$ will relate to $S_i$. All of this is Case~(\ref{it:2}).

    If the witness uses no elements of $\chain$, then there must be $A,B,D\in\F\setminus\chain$ such that $A\supset B$ and $D\sim S_i$ and $\{A,B\}\parallel\{D,S_i\}$. This is Case~(\ref{it:3}).
\end{proof}

Conveniently, Case~(\ref{it:3}) is superfluous, as we demonstrate in Claim~\ref{claim:no_case3}.

\begin{claim}
    If a shackle $S_i \notin \F$ is witnessed in a Case~(\ref{it:3}) fashion by elements $A, B$, and $D$, then each of those elements also witness $S_i$ via Case~(\ref{it:1}) or Case~(\ref{it:2}).
    \label{claim:no_case3}
\end{claim}
\begin{proof}
    Suppose there exist $A,B,D \in \F\setminus\chain$ such that $A \supsetneq B$, $D \sim S_i$, and $\{A,B\} \parallel \{D,S_i\}$.
    
    Using Claim~\ref{claim:two_elements}, if the first bullet point holds for some $j\neq i$, then $C_j\sim S_i$ but since $B\subsetneq C_j\subsetneq A$, either $A\sim S_i$ or $B\sim S_i$, a contradiction. If the first bullet point holds and $j=i$, then $A\sim C_i$ and recall that $D\sim S_i$. Observe that $D\parallel C_i$ otherwise $D\sim A$ or $D\sim B$. As a result, $\{A,C_i\}\parallel\{D,S_i\}$ and $\{B,C_i\}\parallel\{D,S_i\}$ are both instances of Case~(\ref{it:1}).
    
    If the second bullet point holds and $j=i$, then $B\subseteq S_i\subseteq A$, which is a contradiction to the supposition of Case~(\ref{it:3}). If the second bullet point holds for some $j\neq i$, then $A\sim B$ and $C_j\sim S_i$ and $\{C_j,S_i\}\parallel\{B,A\}$, which is an instance of Case~(\ref{it:2}). Moreover, we will have either $B \subsetneq C_i$ or $A \supsetneq C_i$, meaning that either $\{A,C_i\}\parallel\{D,S_i\}$ or $\{B,C_i\}\parallel\{D,S_i\}$ is a Case~(\ref{it:1}).

    In conclusion, if an element $X \in \F$ witnesses $S_i$ as part of a Case~(\ref{it:3}) configuration, then $X$ will also witness $S_i$ as a Case~(\ref{it:1}) or Case~(\ref{it:2}).
\end{proof}

\begin{figure}[h]
\hfill
\begin{subfigure}[t]{0.4\textwidth}
\begin{tikzpicture}[scale=1,vertex/.style={draw=black, very thick, fill=white, circle, minimum width=2pt, inner sep=2pt, outer sep=1pt},edge/.style={very thick}]
    \coordinate (zero) at (0,0);
    \begin{pgfonlayer}{foreground}
    \node[vertex] (0) at ($(zero)+(0,-2)$) {};
    \node[left,xshift=-2pt] at (0) {$\emptyset$};
    \node[vertex] (1) at ($(zero)+(0,2)$) {};
    \node[left,xshift=-2pt] at (1) {$[n]$};
    \node[vertex] (A) at ($(zero) +(0.55,1)$) {};
    \node[above,yshift=2pt] at (A) {$A$};
    \node[vertex] (B) at ($(zero)+(1.5,1)$) {};
    \node[above,yshift=2pt] at (B) {$B$};
    \node[vertex] (C) at ($(zero) + (-0.6,0)$) {};
    \node[left,xshift=-3pt] at (C) {$C_i$};
    \node[vertex] (S) at ($(zero) + (0.2,0)$) {};
    \node[right,xshift=3pt] at (S) {$S_i$};
    \end{pgfonlayer}
    \begin{pgfonlayer}{background}
    \draw[edge, line width = 2.8pt] (0) to [out = 120, in = 275] (C);
    \draw[edge, line width = 2.8pt] (C) to [out = 85, in = 240] (1);
    \draw[edge] (C) -- (A);
    \draw[edge] (S) -- (B);
    \end{pgfonlayer}
\end{tikzpicture}
\caption{Case ~(\ref{it:1}), witnessing $S_i$. Alternatively, $A\subsetneq C_i$ and/or $B\subsetneq S_i$ are possible.}
\label{fig:case1}
\end{subfigure}
\hfill
\begin{subfigure}[t]{0.4\textwidth}
\begin{tikzpicture}[scale=1,vertex/.style={draw=black, very thick, fill=white, circle, minimum width=2pt, inner sep=2pt, outer sep=1pt},edge/.style={very thick}]
    \coordinate (one) at (0,0);
    \begin{pgfonlayer}{foreground}

    \node[vertex] (02) at ($(one)+(0,-2)$) {};
    \node[left,xshift=-2pt] at (02) {$\emptyset$};
    \node[vertex] (12) at ($(one)+(0,2)$) {};
    \node[left,xshift=-2pt] at (12) {$[n]$};
    \node[vertex] (A2) at ($(one)+(1.5,1)$) {};
    \node[above,yshift=2pt] at (A2) {$A$};
    \node[vertex] (B2) at ($(one) +(1.5,-1)$) {};
    \node[below,yshift=-2pt] at (B2) {$B$};
    \node[vertex] (C2) at ($(one) + (-0.6,0)$) {};
    \node[left,xshift=-3pt] at (C2) {$C_{j}$};
    \node[vertex] (Cplus) at ($(one) + (-0.55,0.6)$) {};
    \node[left,xshift=-3pt] at (Cplus) {$C_{j+1}$};
    \node[vertex] (Cminus) at ($(one) + (-0.55,-0.6)$) {};
    \node[left,xshift=-3pt] at (Cminus) {$C_{j-1}$};
    \node[vertex] (S2) at ($(one) + (0.2,0)$) {};
    \node[right,xshift=3pt] at (S2) {$S_{j}$};

    \node[vertex] (Splus) at ($(one) + (0.2,0.6)$){};
    \node[above,yshift = 2pt] at (Splus) {$S_{j+1}$};
    \node[vertex] (Sminus) at ($(one) + (0.2,-0.6)$){};
    \node[below,yshift = -2pt] at (Sminus) {$S_{j-1}$};
    \end{pgfonlayer}
    \begin{pgfonlayer}{background}

    \draw[edge, line width = 2.8pt] (02) to [out = 120, in = 275] (C2);
    \draw[edge, line width = 2.8pt] (C2) to [out = 85, in = 240] (12);

    \draw[edge] (S2) -- (A2);
    \draw[edge, line width = 2.5pt] (B2) -- (A2);
    \draw[edge] (B2) -- (S2);
    \draw[edge] (Cplus) -- (S2);
    \draw[edge] (Cminus) -- (S2);
    
        \draw[edge, line width = 2.8pt] (02) to [out = 120, in = 280] (Cminus);
        \draw[edge, line width = 2.8pt] (Cminus) to [out = 98, in = 273] (C2);
        \draw[edge, line width = 2.8pt] (C2) to [out = 87, in = 263] (Cplus);
    \draw[edge, line width = 2.8pt] (Cplus) to [out = 80, in = 240] (12);
    \draw[edge, line width = 2.5pt] (Splus) -- (C2);
    \draw[edge, line width = 2.5pt] (Sminus) -- (C2);
        
    \end{pgfonlayer}
\end{tikzpicture}
\caption{Case ~(\ref{it:2}), witnessing both $S_{j-1}$ and $S_{j+1}$. Alternatively, $A = S_j$ or $B = S_j$ are possible.}
\label{fig:case2}
\end{subfigure}
\hfill
\caption{The two main ways to witness shackles.}
\end{figure}
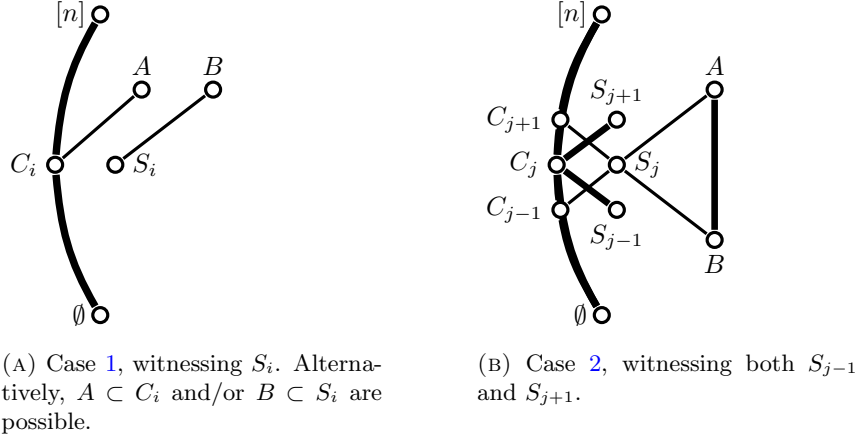

Finally, we establish that a member of $\F\setminus\chain$ cannot be a witness too often. 
\begin{claim}
    Let $F\in\F\setminus\chain$. If $F$ is a shackle, then $F$ witnesses at most two other shackles. If $F$ is not a shackle, then it witnesses at most four shackles.
    \label{claim:twofour}
\end{claim}

\begin{proof}
    Let $F$ be a shackle, specifically $F=S_k$. Recall the case breakdown from Claim~\ref{cl:case breakdown}.
    \begin{itemize}
        \item If $F$ witnesses $S_i$ in a Case~(\ref{it:1}) configuration, then it must be the $A$ set (see Figure~\ref{fig:case1}) because $F\sim C_i$. Furthermore, either $i=k-1$ or $i=k+1$, otherwise $F\sim S_i$. Thus $F$ can witness only $S_{k+1}$ or $S_{k-1}$ in this case.

        \item If $F$ witnesses $S_i$ in a Case~(\ref{it:2}) configuration, then either $i=k-1$ or $i=k+1$, otherwise $F\sim S_i$. In either case, $F$ could play the $A$ or $B$ role (see Figure~\ref{fig:case2}). Thus $F$ can witness only $S_{k+1}$ or $S_{k-1}$ in this case.

        \item If $F$ witnesses $S_i$ in a Case~(\ref{it:3}) configuration, then Claim~\ref{claim:no_case3} ensures that one of the previous two bullet points applies.
    \end{itemize}

    Consequently, if $F$ is a shackle, it can only participate as a witness for $S_{k+1}$ or $S_{k-1}$. 

    Now let $F$ not be a shackle. In particular, let $q$ be the least integer such that $F\subsetneq C_q$ and let $p$ be the greatest integer such that $F\supsetneq C_p$. Since $F$ is not a shackle, then $q\geq p+3$. Note that $F\subsetneq S_{q+1},S_{q+2},\ldots,S_{n-1}$ and $F\supsetneq S_{p-1},S_{p-2},\ldots,S_1$. 

    \begin{itemize}
        \item In a Case~(\ref{it:1}) configuration, as an $A$ vertex, $F$ can only witness the shackles $S_p$ and $S_q$ because it is related to each of $S_1,\ldots,S_{p-1},S_{q+1},\ldots,S_{n-1}$ and it is not related to each of $C_{p+1},\ldots,C_{q-1}$. 

        \item In a Case~(\ref{it:1}) configuration, as a $B$ vertex, $F$ can only witness the shackles $S_{p+1}$ and $S_{q-1}$ (and then, only if $F\supsetneq S_{p+1}$ or $F\subsetneq S_{q-1}$, respectively) because $F$ is related to $C_1,\ldots,C_{p},C_{q},\ldots,C_{n-1}$ and it is not related to each of $S_{p+2},\ldots,S_{q-2}$. 

        \item In a Case~(\ref{it:2}) configuration, as an $A$ vertex, $F$ can only witness the shackles $S_p$ and $S_{p+2}$. This is because the $A$ vertex in Figure~\ref{fig:case2} has $C_{j-1}$ as its largest subset in $\chain$ and in that case can only witness $S_{j-1}=S_{p}$ and $S_{j+1}=S_{p+2}$. 
    
        \item In a Case~(\ref{it:2}) configuration, as a $B$ vertex, $F$ can only witness the shackles $S_q$ and $S_{q-2}$. This is because the $B$ vertex in Figure~\ref{fig:case2} has $C_{j+1}$ as its smallest superset in $\chain$ and in that case can only witness $S_{j+1}=S_{q}$ and $S_{j-1}=S_{q-2}$.

        \item In a Case~(\ref{it:3}) configuration, then Claim~\ref{claim:no_case3} ensures that one of the previous four bullet points applies.
    \end{itemize}
    
    Consequently if $F$ is not a shackle, it can only participate as a witness for a member of the set $\{S_{p},S_{p+1},S_{p+2},S_{q-2},S_{q-1},S_{q}\}$.

    Furthermore, we will show that at most one member of $\{S_{p+1},S_{p+2}\}$ and at most one member of $\{S_{q-1},S_{q-2}\}$ are witnessed by $F$:

        If $S_{p+2}$ is witnessed by $F$, then $F$ is part of a Case~(\ref{it:2}) configuration as an $A$ vertex, and there exists some vertex $G \in \F \setminus \chain$ that acts as the $B$, as shown in Figure~\ref{fig:Fstruct}. If $F$ also witnesses $S_{p+1}$, then $F$ is part of a Case~(\ref{it:1}) configuration as a $B$ vertex, and there exists some vertex $H \in \F \setminus \chain$ with $H \sim C_{p+1}$ acting as the $A$ vertex.
        
        If $H$ is unrelated to $G$, then $\{C_{p+1},H\}\parallel\{F,G\}$ forms a $2C_2$ in $\F$, a contradiction. Therefore $H$ is related to $G$. If $H \subsetneq G$, we then have $H \subsetneq F$, a contradiction. Therefore $H \supsetneq G$. Because $G$ acts as the $B$ vertex in the Case~(\ref{it:2}) configuration, it is a subset of $S_{p+1}$, but not a subset of $C_p$. Thus we have $p+2 \in G$, and therefore $p+2 \in H$.
        
    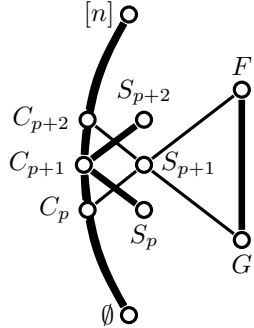
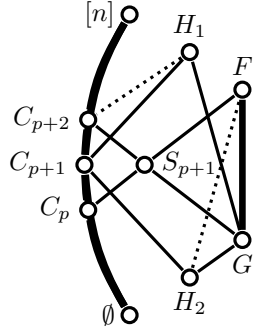
\begin{figure}[h]\hfill
        \begin{subfigure}[t]{0.4\textwidth} \vskip 0pt
            \begin{tikzpicture}[scale=1,vertex/.style={draw=black, very thick, fill=white, circle, minimum width=2pt, inner sep=2pt, outer sep=1pt},edge/.style={very thick}]
                \coordinate (one) at (0,0);
                \begin{pgfonlayer}{foreground}
                    \node[vertex] (02) at ($(one)+(0,-2)$) {};
                    \node[left,xshift=-2pt] at (02) {$\emptyset$};
                    \node[vertex] (12) at ($(one)+(0,2)$) {};
                    \node[left,xshift=-2pt] at (12) {$[n]$};
                    \node[vertex] (A2) at ($(one)+(1.5,1)$) {};
                    \node[above,yshift=2pt] at (A2) {$F$};
                    \node[vertex] (B2) at ($(one) +(1.5,-1)$) {};
                    \node[below,yshift=-2pt] at (B2) {$G$};
                    \node[vertex] (C2) at ($(one) + (-0.6,0)$) {};
                    \node[left,xshift=-3pt] at (C2) {$C_{p+1}$};
                    \node[vertex] (Cplus) at ($(one) + (-0.55,0.6)$) {};
                    \node[left,xshift=-3pt] at (Cplus) {$C_{p+2}$};
                    \node[vertex] (Cminus) at ($(one) + (-0.55,-0.6)$) {};
                    \node[left,xshift=-3pt] at (Cminus) {$C_p$};
                    
                    \node[vertex] (S2) at ($(one) + (0.2,0)$) {};
                    \node[right,xshift=3pt] at (S2) {$S_{p+1}$};
                    \node[vertex] (Splus) at ($(one) + (0.2,0.6)$){};
                    \node[above,yshift = 2pt] at (Splus) {$S_{p+2}$};
                    \node[vertex] (Sminus) at ($(one) + (0.2,-0.6)$){};
                    \node[below,yshift = -2pt] at (Sminus) {$S_{p}$};
                \end{pgfonlayer}
                \begin{pgfonlayer}{background}
                    \draw[edge, line width = 2.8pt] (02) to [out = 120, in = 275] (C2);
                    \draw[edge, line width = 2.8pt] (C2) to [out = 85, in = 240] (12);
                
                    \draw[edge] (S2) -- (A2);
                    \draw[edge, line width = 2.5pt] (B2) -- (A2);
                    \draw[edge] (B2) -- (S2);
                    \draw[edge] (Cplus) -- (S2);
                    \draw[edge] (Cminus) -- (S2);
    
                    \draw[edge, line width = 2.8pt] (02) to [out = 120, in = 280] (Cminus);
                    \draw[edge, line width = 2.8pt] (Cminus) to [out = 98, in = 273] (C2);
                    \draw[edge, line width = 2.8pt] (C2) to [out = 87, in = 263] (Cplus);
                    \draw[edge, line width = 2.8pt] (Cplus) to [out = 80, in = 240] (12);
                    \draw[edge, line width = 2.5pt] (Splus) -- (C2);
                    \draw[edge, line width = 2.5pt] (Sminus) -- (C2);
                \end{pgfonlayer}
            \end{tikzpicture}
            \caption{For F to witness $S_{p+2}$ requires this structure. It will therefore also witness $S_p$.}
            \label{fig:Fstruct}
        \end{subfigure}\hfill
        \begin{subfigure}[t]{0.4\textwidth} \vskip 0pt
            \begin{tikzpicture}[scale=1,vertex/.style={draw=black, very thick, fill=white, circle, minimum width=2pt, inner sep=2pt, outer sep=1pt},edge/.style={very thick}]
                \coordinate (one) at (0,0);
                \begin{pgfonlayer}{foreground}
                    \node[vertex] (02) at ($(one)+(0,-2)$) {};
                    \node[left,xshift=-2pt] at (02) {$\emptyset$};
                    \node[vertex] (12) at ($(one)+(0,2)$) {};
                    \node[left,xshift=-2pt] at (12) {$[n]$};
                    \node[vertex] (A2) at ($(one)+(1.5,1)$) {};
                    \node[above,yshift=2pt] at (A2) {$F$};
                    \node[vertex] (B2) at ($(one) +(1.5,-1)$) {};
                    \node[below,yshift=-2pt] at (B2) {$G$};
                    \node[vertex] (C2) at ($(one) + (-0.6,0)$) {};
                    \node[left,xshift=-3pt] at (C2) {$C_{p+1}$};
                    \node[vertex] (Cplus) at ($(one) + (-0.55,0.6)$) {};
                    \node[left,xshift=-3pt] at (Cplus) {$C_{p+2}$};
                    \node[vertex] (Cminus) at ($(one) + (-0.55,-0.6)$) {};
                    \node[left,xshift=-3pt] at (Cminus) {$C_p$};
                    
                    \node[vertex] (S2) at ($(one) + (0.2,0)$) {};
                    \node[right,xshift=3pt] at (S2) {$S_{p+1}$};

                    \node[vertex] (H1) at ($(one) + (0.8,1.5)$) {};
                    \node[above,yshift=2pt] at (H1) {$H_1$};
                    \node[vertex] (H2) at ($(one) + (0.8,-1.5)$) {};
                    \node[below,yshift=-2pt] at (H2) {$H_2$};
                \end{pgfonlayer}
                \begin{pgfonlayer}{background}
                    \draw[edge, line width = 2.8pt] (02) to [out = 120, in = 275] (C2);
                    \draw[edge, line width = 2.8pt] (C2) to [out = 85, in = 240] (12);
                
                    \draw[edge] (S2) -- (A2);
                    \draw[edge, line width = 2.5pt] (B2) -- (A2);
                    \draw[edge] (B2) -- (S2);
                    \draw[edge] (Cplus) -- (S2);
                    \draw[edge] (Cminus) -- (S2);
    
                    \draw[edge, line width = 2.8pt] (02) to [out = 120, in = 280] (Cminus);
                    \draw[edge, line width = 2.8pt] (Cminus) to [out = 98, in = 273] (C2);
                    \draw[edge, line width = 2.8pt] (C2) to [out = 87, in = 263] (Cplus);
                    \draw[edge, line width = 2.8pt] (Cplus) to [out = 80, in = 240] (12);
                    \draw[edge] (H2) to (B2);
                    \draw[edge] (H2) to (C2);

                    \draw[edge] (H1) to (B2);
                    \draw[edge] (H1) to (C2);

                    \draw[edge,style=dotted] (H1) to (Cplus);
                    \draw[edge,style=dotted] (H2) to (A2);
                \end{pgfonlayer}
            \end{tikzpicture}
            \caption{For F to witness $S_{p+1}$ requires $H$ to be related to $G$, but either relation between $H$ and $G$ results in a contradiction. Here we illustrate the different possibilities for $H$. The case where $H\supsetneq G$ is $H_1$ and the case where $H\subsetneq G$ is $H_2$.}
            \label{fig:Gstruct}
        \end{subfigure}
        \hfill
        \caption{F cannot witness both $S_{p+1}$ and $S_{p+2}$}
        \label{fig:enter-label}
    \end{figure}
        We know that $H$ is related to $C_{p+1}$, as seen in Figure ~\ref{fig:Gstruct}. Because $p+2 \in H$, we then know that it cannot be a subset of $C_{p+1}$, and is instead a superset of $C_{p+1}$. Therefore $H \supsetneq C_{p+1} \cup \{p+2\} = C_{p+2} \supsetneq S_{p+1}$, which is a contradiction, as $H$ must not relate to $S_{p+1}$ since it is acting as the $A$ vertex in a Case~(\ref{it:1}) configuration. Therefore $H$ cannot exist, and we see that $F$ cannot witness both $S_{p+1}$ and $S_{p+2}$.

    By a similar argument, inverting the direction of the relations, $F$ cannot witness both $S_{q-1}$ and $S_{q-2}$. This concludes the proof of Claim ~\ref{claim:twofour}.    
\end{proof}

Let $s$ be the number of shackles present in $\F$, and let $k$ be the total number of non-chain elements in $\F$. Every one of the $n-1$ shackles must either be in $\F$, or be witnessed by at least 2 elements in $\F$, so we have:
\begin{align*}
    n-1 &\leq s + (1/2)\bigl(2s + 4(k-s)\bigr) \\
    n-1 &\leq 2k\\
    n/2 - 1/2 &\leq k\\
\end{align*}

Adding $k$ to the $n+1$ elements in the chain yields $3n/2 + 1/2$ as the lower bound of $\sat^*(n,2C_2)$.

\section{Examples of $2C_2$-saturated families in $\power([n])$}
\label{sec:examples}

For $n\leq 6$, we have calculated, via exhaustive computer search, the number of saturating families. Up to $n=6$, we can confirm that $2n$ is the smallest possible size for a saturating set. We first use Theorem~\ref{thm:KLMPP} to establish that the family has a maximal chain, specifically the chain $\chain=\bigl\{\emptyset, \{1\}, \{1,2\}, \ldots, \{1,2,\ldots,n-1\}, [n]\bigr\}$. With those $n+1$ elements, we then pair a family $\F$ of size $2n$ with its dual (that is, we turn the Boolean lattice ``upside down'') to arrive at an isomorphic copy of $\F$. For $n\leq 6$ these are distinct families except in one case when $n$ is odd. 

For $n=3$, there are $5$ distinct induced $2C_2$ saturated families of size $2n$ (3 isomorphism classes). For $n=4$ there are $18$ with $9$ isomorphism classes. For $n=5$, there are $83$ with $42$ isomorphism classes. For $n=6$, there are $452$ with $226$ isomorphism classes. Fixing the chain $\chain$ establishes a linear order of the $n$ numbers by the elements in $\chain$ in which they first appear, making isomorphisms possible only through duality. All families up to $n=5$ are listed below in Appendix~\ref{sec:constructions}.

We note that Keszegh et al.~\cite{KLMPP} established an isomorphism class of induced-$2C_2$-saturated families of size $2n$ in $\power([n])$. Specifically, it consists of the chain $\chain$ and the singletons $\{2\}, \{3\}, \ldots, \{n\}$. Note that this is isomorphic to $\chain$ plus the anti-singletons $\overline{\{2\}}, \overline{\{3\}}, \ldots, \overline{\{n\}}$.

Inspired by the computer output, we were able to construct an infinite collection of additional families of induced-$2C_2$-saturated families of size $2n$. Specifically, the collection of families $\F_i^*$ consisting of the chain $\chain$ plus the sets 
\begin{align*}
    \{2\},\{3\},\ldots,\{i\},\overline{\{i\}},\overline{\{i+1\}},\ldots,\overline{\{n-1\}} 
\end{align*}
for some integer $2 \leq i \leq n-1$.

 In particular, this family is self-dual when $n$ is odd and $i = (n+1)/2$.
\begin{prop}
    The above family $\F_i^*$ is induced-$2C_2$-saturated. 
\end{prop}

\begin{proof}
We first show that $\F_i^*$ is induced-$2C_2$-free. Observe that the Hasse diagram of singletons and anti-singletons form a complete bipartite graph except for a missing edge between $\{i\}$ and $\overline{\{i\}}$. We also note that the chain element $C_{i}$ is a superset of all of the singletons, and the chain element $C_{i-1}$ is a subset of all of the anti-singletons.

If an induced copy of $2C_2$ has two elements of the chain $\chain$, then they must be related to each other, and be unrelated to the other two elements. However, any chain element $C_j$ must be related to either all the singletons in $\F_i^*$, or all the anti-singletons in $\F_i^*$, and the only elements unrelated to $C_j$ will form an antichain. So an induced copy of $2C_2$ can have at most one chain element.

Therefore, at least 3 elements come from the non-chain elements of $\F_i^*$. 
This is only possible if two of those three elements are $\{i\}$ and $\overline{\{i\}}$. Now we need a paired element for each one. However, the elements of $\F_i^*$ that relate to $\{i\}$ all relate to every element that is related to $\overline{\{i\}}$. So it is impossible to create a $2C_2$ within $\F_i^*$.

Now we show that $\F_i^*$ is $2C_2$ saturated. Consider any element $X\notin\F_i^*$. Either $X\supset\{i\}$ or $X\subset\overline{\{i\}}$. 

If $X\supset\{i\}$, and there exists a singleton $\{j\}$ in $\F_i^*$ that is unrelated to $X$, then the four elements $\bigl\{\{i\}, X\bigr\} \parallel \bigl\{\{j\}, \overline{\{i\}}\bigr\}$ form a $2C_2$.

If $X\supset\{i\}$ and is also a superset of every other singleton in $\F_i^*$ (hence $X\supseteq [i]$), then take $j$ to be the smallest number that is not contained in $X$, and take $k > j$ to be a number that is contained in $X$. Those $j$ and $k$ must exist because $X\not\in\chain$.  Then the four elements $\bigl\{C_j, \overline{\{k\}}\bigr\} \parallel \bigl\{X, \overline{\{j\}}\bigr\}$ forms a $2C_2$.

If $X\subset\overline{\{i\} }$, then we follow the same argument, but in the dual.
\end{proof}

\section{Conclusions}
\label{sec:conc}
The observation that shackles are important to identifying induced-$2C_2$-saturation was important for our main result but it may be useful for improving the lower bound. For example, for $n\leq 6$ we observed that if $\F$ is an induced-$2C_2$-saturated family of size $2n$ and $\F\setminus\chain$ has no shackle, then $\F\setminus\chain$ forms an antichain. However, we suspect that this will not hold for larger values of $n$. 

Based on the data we have accumulated, we believe that $\sat^*(n,2C_2)=2n$ for all values of $n$. Moreover, we believe there are many nonisomorphic families that achieve this bound. 

\section{Acknowledgments}
Martin's research was partially supported by a grant from the Simons Foundation \#709641 and this research was partially done while Martin was on an MTA Distinguished Guest Scientist Fellowship 2023 at the HUN-REN Alfr\'ed R\'enyi Institute of Mathematics. Veldt's research was supported by a grant from the National Science Foundation DMS-1839918 (RTG). Both authors are indebted to the Alfr\'ed R\'enyi Institute of Mathematics for hosting Veldt for a collaboration visit. They would also like to thank Bal\'azs Patk\'os for helpful comments on the manuscript. In addition, we are indebted to two anonymous referees for their very careful reading and identifying errors and improving the clarity of the manuscript. 

\section{Conflict of Interest and Data Availability Statements}
On behalf of all authors, the corresponding author states that there is no conflict of interest.

Relevant data is provided in Appendix~\ref{sec:constructions}. The Python script used to generate it is available upon request.

\newpage
\bibliographystyle{alpha}
\bibliography{bibli}
\newpage
\appendix
\section{Known saturating sets for $n \leq 5$}
\label{sec:constructions}
In this section we give the results from the computer output of the families that are induced-$2C_2$-saturated in $\power\bigl([n]\bigr)$ and are of size $2n$. Each such family contains the maximum chain $\chain=\bigl\{\emptyset, [1], [2],\ldots,[n-1],[n]\bigr\}$. What follows is the enumeration of the other sets. For shorthand, we denote, for example, the set $\{2,3,5\}$ by $235$. \\~\\


\noindent When $n = 3$: \\~\\ 
{\footnotesize
\begin{tabular}{|c|c|}
\hline
  Set & Dual \\ \hline
  $\{2,3\}$ & $\{13,23\}$  \\ \hline
  $\{2,23\}$ & $\{3,13\}$  \\ \hline
  $\{2,13\}$ & Itself \\ \hline
\end{tabular}
}
\vfill

\noindent When $n = 4$: \\~\\
{\footnotesize
\begin{tabular}{|c|c||c|c|}\hline
    Set & Dual & Set & Dual\\ \hline
    $\{2,3,4\}$ & $\{124,134,234\}$ & $\{2,3,14\}$ & $\{124,134,23\}$\\ \hline
    $\{2,3,24\}$ & $\{124,134,24\}$ & $\{2,3,124\}$ & $\{124,134,2\}$\\ \hline
    $\{2,4,13\}$ & $\{124,234,13\}$ & $\{2,4,23\}$ & $\{124,234,14\}$\\ \hline
    $\{2,4,124\}$ & $\{124,234,2\}$ & $\{3,4,13\}$ & $\{134,234,13\}$\\ \hline
    $\{4,13,23\}$ & $\{13,14,234\}$ & &\\ \hline 
\end{tabular}
}
\vfill
\noindent When $n = 5$: \\~\\
{\footnotesize
\begin{tabular}{|c|c||c|c|}\hline
    Set & Dual & Set & Dual\\ \hline
    $\{2,3,4,5\}$ & $\{1235,1245,1345,2345\}$ &$\{2,3,4,15\}$ & $\{234,1235,1245,1345\}$\\ \hline
    $\{2,3,4,25\}$ & $\{235,1235,1245,1345\}$ &$\{2,3,4,125\}$ & $\{23,1235,1245,1345\}$ \\ \hline
    $\{2,3,4,1235\}$ & $\{2,1235,1245,1345\}$ &$\{2,3,5,14\}$& $\{134,1235,1245,2345\}$ \\ \hline
    $\{2,3,5,24\}$ & $\{135,1235,1245,2345\}$ & $\{2,3,5,124\}$ & $\{13,1235,1245,2345\}$\\ \hline
    $\{2,3,5,1235\}$ & $\{2,1235,1245,2345\}$ & $\{2,3,1235,1245\}$ & Itself \\ \hline
    $\{2,4,5,13\}$ & $\{124,1235,1345,2345\}$ & $\{2,4,5,23\}$ & $\{125,1235,1345,234\}$ \\ \hline
    $\{2,4,5,124\}$ & $\{13,1235,1345,2345\}$ & $\{2,4,13,1235\}$ & $\{2,124,1235,1345\}$ \\ \hline
    $\{2,4,23,1235\}$ & $\{2,125,1235,1345\}$ & $\{2,5,13,1235\}$ & $\{2,124,1235,2345\}$ \\ \hline
    $\{2,5,23,1235\}$ & $\{2,125,1235,2345\}$ & $\{2,5,124,134\}$ & $\{13,14,1235,2345\}$ \\ \hline
    $\{2,5,124,234\}$ & $\{13,15,1235,2345\}$ & $\{2,124,125,1345\}$ & $\{4,13,23,1235\}$ \\ \hline
    $\{2,124,125,2345\}$ & $\{5,13,23,1235\}$ & $\{3,4,5,13\}$ & $\{124,1245,1345,2345\}$ \\ \hline
    $\{3,4,13,1235\}$ & $\{2,124,1245,1345\}$ & $\{3,5,13,1235\}$ & $\{2,124,1245,2345\}$ \\ \hline
    $\{3,14,125,245\}$ & $\{23,34,135,1245\}$ & $\{3,15,124,245\}$ & $\{13,34,234,1245\}$ \\ \hline
    $\{3,24,125,145\}$ & $\{23,35,134,1245\}$ & $\{3,25,124,145\}$ & $\{13,35,235,1245\}$\\ \hline
    $\{4,5,13,23\}$ & $\{124,125,1345,2345\}$ & $\{4,13,15,235\}$ & $\{25,124,234,1345\}$ \\ \hline
    $\{4,13,125,235\}$ & $\{23,25,124,1345\}$ & $\{4,13,135,235\}$ & $\{24,25,124,1345\}$ \\ \hline
    $\{4,15,124,235\}$ & $\{13,25,234,1345\}$ & $\{4,23,125,135\}$ & $\{23,24,125,1345\}$ \\ \hline
    $\{4,25,124,135\}$ & $\{13,24,235,1345\}$ & $\{5,13,14,234\}$ & $\{15,124,134,2345\}$ \\ \hline
    $\{5,13,124,234\}$ & $\{13,15,124,2345\}$ & $\{5,13,134,234\}$ & $\{14,15,124,2345\}$ \\ \hline
    $\{5,14,124,234\}$ & $\{13,15,134,2345\}$ & $\{5,23,124,134\}$ & $\{13,14,125,2345\}$ \\ \hline
    $\{5,24,124,134\}$ & $\{13,14,135,2345\}$ & $\{5,124,134,234\}$ & $\{13,14,15,2345\}$ \\ \hline
\end{tabular}
}
\end{document}